\documentclass[12pt]{article}
\usepackage{amssymb}
\usepackage{amsthm}
\newtheorem{Definition}{Definition}
\newtheorem{Lemma}[Definition]{Lemma}

\newtheorem{Theorem}[Definition]{Theorem}
\newtheorem{Remark}[Definition]{Remark}

\title{Coupled right orthosemirings induced by orthomodular lattices\thanks{This is a post-peer-review, pre-copyedit version of an article published in Order. The final authenticated version is available online at: {\texttt http:/dx.doi.org/10.1007/s11083-015-9383-7.}}}
\author{Ivan~Chajda and Helmut~L\"anger}
\date{}
\begin{document}
\footnotetext[1]{Support of the research of both authors by the Austrian Science Fund (FWF), project I~1923-N25, and the Czech Science Foundation (GA\v CR), project 15-34697L, as well as by the project entitled "Ordered structures for algebraic logic", supported by AKTION Austria -- Czech Republic, project 71p3, is gratefully acknowledged.}
\maketitle
\begin{abstract}
L.~P.~Belluce, A.~Di~Nola and B.~Gerla established a connection between MV-algebras and (dually) lattice ordered semirings by means of so-called coupled semirings. A similar connection was found for basic algebras and semilattice ordered right near semirings by the authors. The aim of this paper is to derive an analogous connection for orthomodular lattices and certain semilattice ordered near semirings via so-called coupled right orthosemirings.
\end{abstract} 

{\bf AMS Subject Classification:} 06C15, 06F25, 16Y30, 03G25

{\bf Keywords:} orthomodular lattice, right near semiring, semilattice ordered right near semiring, coupled right orthosemiring

It is well-known that MV-algebras play a crucial role in the algebraic axiomatization of so-called many-valued Lukasiewicz logics. This is the reason why MV-algebras were intensively studied in the last decades. The connection between MV-algebras and certain semirings was developed by A.~Di~Nola and B.~Gerla (\cite{DG} and \cite{Ge}). They recognized that to every MV-algebra there can be assigned a certain triplet which consists of two semirings, one lattice ordered and the other one dually lattice ordered, and an involutive isomorphism between them. Although this construction which gives a full characterization is interesting, it is more remarkable that, with slight modifications, the mentioned construction can be used also for basic algebras and commutative basic algebras (see \cite{CL1} and \cite{CL2}).

Orthomodular lattices play a similar role for the logic of quantum mechanics as MV-algebras do for Lukasiewicz's many-valued logic. In fact, orthomodular lattices originated in the 1930's where G.~Birkhoff and J.~von~Neumann used them in order to describe quantum events. They considered certain operators in Hilbert spaces and the lattice of closed subspaces of a Hilbert space. For some details, the reader is referred to \cite B and \cite K. Hence, the natural question arises if a similar construction to that in \cite{DG} and \cite{Ge} using semilattice-like structures can be used in order to find a characterization of orthomodular lattices. Surprisingly, this is not only possible but we can use the machinery derived in \cite{CL2} with very small modifications. The resulting structure which fully characterizes orthomodular lattices is called a coupled right orthosemiring which is analogous to the coupled semirings used in \cite{BDF}, \cite{DG} and \cite{Ge}. Hence we will show that our way of representing algebras which are axiomatizations of certain propositional logics is universal in the sense that it does not depend on the fact that the underlying logic is a many-valued logic or the logic of quantum mechanics. Both constructions are similar to each other and differ only in the underlying axioms.

In fact, we will use only very elementary concepts of orthomodular lattices. The reader is referred to the monographs \cite B and \cite K. For some algebraic aspects of orthomodular lattices see e.~g.\ \cite{BH}.

Concerning semilattice-like structures, we use some concepts for semilattices, see \cite{Go} and \cite{KS}. Some concepts concerning lattice ordered semirings or coupled semirings are taken from \cite{CL1}, \cite{DG} and \cite{Ge}, whereas concepts concerning semilattice ordered right near semirings and coupled near semirings can be found in \cite{CL2}. For the reader's convenience, all these concepts are repeated here. Our paper \cite{CL2} also contains some results on right near semirings which are not used here.

We start with the definition of an orthomodular lattice.

\begin{Definition}\label{def1}
An {\em orthomodular lattice} {\rm(}see e.~g.\ {\rm\cite B} and {\rm\cite K)} is an algebra ${\mathcal L}=(L,\vee,\wedge,',$ $0,1)$ of type $(2,2,1,0,0)$ satisfying the following axioms:
\begin{enumerate}
\item[{\rm(i)}] $(L,\vee,\wedge,0,1)$ is a bounded lattice.
\item[{\rm(ii)}] $x\leq y$ implies $x'\geq y'$.
\item[{\rm(iii)}] $(x')'=x$
\item[{\rm(iv)}] $x\leq y$ implies $y=x\vee(y\wedge x')$.
\end{enumerate}
\end{Definition}

\begin{Remark}\label{rem2}
Condition {\rm(iv)} is called the {\em orthomodular law}. From this law it follows that $x\vee x'=x\vee(1\wedge x')=1$, i.~e.\ $x'$ is a lattice-theoretical complement of $x$.
\end{Remark}

In every orthomodular lattice there can be defined the so-called commutation relation:

\begin{Definition}\label{def2}
Let ${\mathcal L}=(L,\vee,\wedge,',$ $0,1)$ be an orthomodular lattice and $a,b\in L$. The elements $a$ and $b$ are said to commute with each other, in symbols $a$ {\rm C} $b$, if $a=(a\wedge b)\vee(a\wedge b')$.
\end{Definition}

This commutation relation has the following properties:

\begin{Lemma}\label{lem1}
{\rm(}cf.\ {\rm\cite K)} Let ${\mathcal L}=(L,\vee,\wedge,',$ $0,1)$ be an orthomodular lattice and $a,b,c\in L$. Then the following hold:
\begin{enumerate}
\item[{\rm(i)}] If $a$ {\rm C} $b$ then $b$ {\rm C} $a$.
\item[{\rm(ii)}] If $a\leq b$ then $a$ {\rm C} $b$.
\item[{\rm(iii)}] If $a$ {\rm C} $b$ then $a$ {\rm C} $b'$.
\item[{\rm(iv)}] If two of the three elements $a$, $b$ and $c$ commute with the third one then $(a\vee b)\wedge c=(a\wedge c)\vee(b\wedge c)$ and $(a\wedge b)\vee c=(a\vee c)\wedge(b\vee c)$.
\end{enumerate}
\end{Lemma}

The concept of a right near semiring was introduced by the authors in \cite{CL2}. For the reader's convenience we repeat this definition.

\begin{Definition}\label{def3}
A {\em right near semiring} is an algebra ${\mathcal R}=(R,+,\cdot,0,1)$ of type $(2,2,0,0)$ satisfying the following axioms:
\begin{enumerate}
\item[{\rm(i)}] $(R,+,0)$ is a commutative monoid.
\item[{\rm(ii)}] $(R,\cdot)$ is a groupoid with neutral element $1$.
\item[{\rm(iii)}] $(x+y)z=xz+yz$ for all $x,y,z\in R$
\item[{\rm(iv)}] $x0=0x=0$ for all $x\in R$
\end{enumerate}
If, in addition, $\cdot$ is both associative and distributive with respect to $+$ then ${\mathcal R}$ is called a {\em semiring}.
\end{Definition}

Recall from \cite{BDF} or \cite{DG} that an {\em{\rm MV}-algebra} is an algebra $(A,\oplus,\neg,0)$ of type $(2,1,0)$ such that $(A,\oplus,0)$ is a commutative monoid and it satisfies the identities $¬¬x = x$, $x\oplus\neg0=\neg0$ and $\neg(\neg x\oplus y)\oplus y=\neg(\neg y\oplus x)\oplus x$. We put $1:=\neg0$. In the following let $x,y\in A$. It is well-known that every MV-algebra can be considered as a bounded lattice with an antitone involution $(A,\vee,\wedge,\neg,0,1)$ where 
\begin{equation}\label{equ1}
x\vee y:=\neg(\neg x\oplus y)\oplus y\mbox{ and }x\wedge y:=\neg(\neg x\vee\neg y).
\end{equation}
The lattice order $\leq$ is given by
\begin{equation}\label{equ2}
x\leq y\mbox{ if and only if }\neg x\oplus y=1.
\end{equation}
It can be easily checked that
\begin{equation}\label{equ3}
x\leq y\mbox{ if and only if }\neg y\leq\neg x.
\end{equation}
Moreover, a binary operation $\odot$ on $A$ can be defined by
\begin{equation}\label{equ4}
x\odot y:=\neg(\neg x\oplus\neg y).
\end{equation}
It is easy to prove that $(A,\vee,\odot,0,1)$ and $(A,\wedge,\oplus,1,0)$ are semirings. \\
Recall from \cite{BDF} that a semiring ${\mathcal S}=(A,+,\cdot,0,1)$ is called {\em lattice-ordered} if there exists a lattice $(A,\vee,\wedge)$ such that $x+y=x\vee y$ and $xy\leq x\wedge y$ for all $x,y\in A$. Dually, ${\mathcal S}$ is called {\em dually lattice-ordered} if there exists a lattice $(A,\vee,\wedge)$ such that $xy\geq x\vee y$ and $x+y=x\wedge y$ for all $x,y\in A$. It was shown in \cite{Ge} that $(A,\vee,\odot,0,1)$ and $(A,\wedge,\oplus,1,0)$ are lattice-ordered and dually lattice-ordered, respectively. \\
In \cite{Ge} to every MV-algebra ${\mathcal A}=(A,\oplus,\neg,0)$ there was assigned the so-called {\em coupled semiring}
\[
((A,\vee,\odot,0,1), (A,\wedge,\oplus,1,0),\neg)
\]
such that $x\oplus(\neg x\odot y)=x\vee y$ for all $x,y\in A$. \\
If, conversely,
\[
{\mathcal T}=((A,\vee,\cdot,0,1),(A,\wedge,+,1,0),\alpha)
\]
is a triple consisting of a lattice ordered semiring ${\mathcal A}_1:=(A,\vee,\cdot,0,1)$, a dually lattice-ordered semiring ${\mathcal A}_2:=(A,\wedge,+,1,0)$ and an involutive isomorphism $\alpha$ of ${\mathcal A}_1$ onto ${\mathcal A}_2$ such that $x+(\alpha(x)y)=x\vee y$ for all $x,y\in A$ then $(A,+,\alpha,0)$ is an MV-algebra whose coupled semiring is just ${\mathcal T}$. \\
{\em Basic algebras} were introduced by the first author as a generalization of MV-algebras in the sense that the binary operation $\oplus$ need not be neither associative nor commutative. More precisely, a {\em basic algebra} is an algebra $(A,\oplus,\neg,0)$ of type $(2,1,0)$ satisfying the identities
\begin{eqnarray*}
x\oplus0 & = & x, \\
\neg(\neg x) & = & x, \\
\neg(\neg x\oplus y)\oplus y & = & \neg(\neg y\oplus x)\oplus x\mbox{ and} \\
\neg(\neg(\neg(x\oplus y)\oplus y)\oplus z)\oplus(x\oplus z) & = & 1
\end{eqnarray*}
where $1:=\neg0$. \\
The concept of a basic algebra has no connection to the so-called basic logic introduced by P.~Hajek. The algebra connected with Hajek's basic logic is called BL-algebra. \\
In every basic algebra there can be introduced a partial order relation as given by (\ref{equ2}) which becomes a lattice order with respect to the operations defined by (\ref{equ1}). Moreover, it satisfies also (\ref{equ3}). It is known that a basic algebra is an MV-algebra if and only if $\oplus$ is associative (and, in this case, it is also commutative). On the other hand, there exist commutative basic algebras which are not MV-algebras despite the fact that every finite commutative basic algebra is an MV-algebra. \\
This analogy motivated us in our previous papers \cite{CL2} and \cite{CL1} to modify the aforementioned construction of coupled semirings also for basic algebras and commutative basic algebras, respectively. However, this is possible only if the concept of semiring is substituted by that of a right near semiring  and near semiring, respectively. However, the concept of a (dually) lattice-ordered right near semiring has to be modified as follows:

\begin{Definition}\label{def5}
A right near semiring $(R,+,\cdot,0,1)$ is called
\begin{itemize}
\item {\em $\vee$-semilattice ordered} if there exists a join-semilattice operation $\vee$ on $R$ such that $x+y=x\vee y$ and $xy\leq y$ for all $x,y\in R$ with respect to the induced order.
\item {\em $\wedge$-semilattice ordered} if there exists a meet-semilattice operation $\wedge$ on $R$ such that $x+y=x\wedge y$ and $xy\geq y$ for all $x,y\in R$ with respect to the induced order.
\item a {\em near semiring} if $\cdot$ is distributive with respect to $+$.
\item {\em commutative} if $\cdot$ is commutative.
\end{itemize}
\end{Definition}

\begin{Remark}\label{rem1}
Condition~{\rm(iii)} is called the {\em right distributive law}. Of course, every commutative right near semiring is a commutative near semiring {\rm(}in the sense of {\rm\cite{CL1})}.
\end{Remark}

In order to be able to introduce a triple construction of coupled right near semirings for orthomodular lattices we must adapt the corresponding notions as follows and define coupled right orthosemirings:

\begin{Definition}\label{def4}
A {\em coupled right orthosemiring} is an ordered triple
\[
((R,\vee,\cdot,0,1),(R,\wedge,\ast,1,0),\alpha)
\]
satisfying the following conditions:
\begin{enumerate}
\item[{\rm(R1)}] $(R,\vee,\wedge)$ is a lattice.
\item[{\rm(R2)}] $(R,\vee,\cdot,0,1)$ is a $\vee$-semilattice ordered right near semiring.
\item[{\rm(R3)}] $(R,\wedge,\ast,1,0)$ is a $\wedge$-semilattice ordered right near semiring.
\item[{\rm(R4)}] $\alpha$ is an involutive isomorphism between $(R,\vee,\cdot,0,1)$ and $(R,\wedge,\ast,1,0)$.
\item[{\rm(R5)}] $(x\wedge\alpha(y))\vee y=x\ast y$ for all $x,y\in R$
\item[{\rm(R6)}] $y\ast(x\wedge y)=y$ for all $x,y\in R$
\end{enumerate}
\end{Definition}

Now we can assign to every orthomodular lattice a coupled right orthosemiring in some natural way. The term operations occurring in the following theorem are the so-called {\em Sasaki projection} and its dual version (see \cite K).

\begin{Theorem}\label{th1}
Let ${\mathcal L}=(L,\vee,\wedge,',$ $0,1)$ be an orthomodular lattice and define two binary operations $\oplus$ and $\odot$ on $L$ by
\begin{eqnarray*}
x\oplus y & := & (x\wedge y')\vee y\mbox{ and} \\
x\odot y & := & (x\vee y')\wedge y
\end{eqnarray*}
for all $x,y\in L$. Then
\[
{\bf N}({\mathcal L}):=((L,\vee,\odot,0,1),(L,\wedge,\oplus,1,0),')
\]
is a coupled right orthosemiring.
\end{Theorem}

\begin{proof}
Let $a,b,c\in L$. (R1) and (R4) are evident.
\begin{enumerate}
\item[(R2)] Obviously, $(L,\vee,0)$ is a commutative monoid. Moreover we have
\begin{eqnarray*}
a\odot1 & = & (a\vee1')\wedge1=a\vee0=a\mbox{ and} \\
1\odot a & = & (1\vee a')\wedge a=1\wedge a=a.
\end{eqnarray*}
Hence $(L,\odot,1)$ is a groupoid with neutral element. Moreover, using Lemma~\ref{lem1} we obtain
\begin{eqnarray*}
(a\vee b)\odot c & = & ((a\vee b)\vee c')\wedge c=(a\vee b\vee c')\wedge c=((a\vee c')\vee(b\vee c'))\wedge c= \\
& = & ((a\vee c')\wedge c)\vee((b\vee c')\wedge c)=(a\odot c)\vee(b\odot c).
\end{eqnarray*}
Finally,
\begin{eqnarray*}
a\odot0 & = & (a\vee0')\wedge0=0\mbox{ and} \\
0\odot a & = & (0\vee a')\wedge a=a'\wedge a=0.
\end{eqnarray*}
This shows that $(L,\vee,\odot,0,1)$ is a right near semiring. Because of $a\odot b=(a\vee b')\wedge b\leq b$, this right near semiring is $\vee$-semilattice ordered.
\item[(R3)] Obviously, $(L,\wedge,1)$ is a commutative monoid. Moreover we have
\begin{eqnarray*}
a\oplus0 & = & (a\wedge0')\vee0=a\wedge1=a\mbox{ and} \\
0\oplus a & = & (0\wedge a')\vee a=0\vee a=a.
\end{eqnarray*}
Hence $(L,\oplus,0)$ is a groupoid with neutral element. Moreover, using Lemma~\ref{lem1} we obtain
\begin{eqnarray*}
(a\wedge b)\oplus c & = & ((a\wedge b)\wedge c')\vee c=(a\wedge b\wedge c')\vee c=((a\wedge c')\wedge(b\wedge c'))\vee c= \\
& = & ((a\wedge c')\vee c)\wedge((b\wedge c')\vee c)=(a\oplus c)\wedge(b\oplus c).
\end{eqnarray*}
Finally,
\begin{eqnarray*}
a\oplus1 & = & (a\wedge1')\vee1=1\mbox{ and} \\
1\oplus a & = & (1\wedge a')\vee a=a'\vee a=1.
\end{eqnarray*}
This shows that $(L,\wedge,\oplus,1,0)$ is a right near semiring. Because of $a\oplus b=(a\wedge b')\vee b\geq b$, this right near semiring is $\wedge$-semilattice ordered.
\item[(R5)] follows from the definition of $\oplus$.
\item[(R6)] According to orthomodularity we have
\[
b\oplus(a\wedge b)=(b\wedge(a\wedge b)')\vee(a\wedge b)=b.
\]
\end{enumerate}
\end{proof}

Next we are going to prove the converse of Theorem~\ref{th1}, i.~e.\ we will show that to each coupled right orthosemiring we can assign an orthomodular lattice in a natural way.

\begin{Theorem}\label{th2}
Let ${\mathcal N}=((R,\vee,\cdot,0,1),(R,\wedge,\ast,1,0),\alpha)$ be a coupled right orthosemiring. Then
\[
{\bf L}({\mathcal N}):=(R,\vee,\wedge,\alpha,0,1)
\]
is an orthomodular lattice.
\end{Theorem}

\begin{proof}
Let $a,b\in R$. According to (R1), $(R,\vee,\wedge)$ is a lattice. Because of (R2) and (R3), $(R,\vee,0)$ and $(R,\wedge,1)$ are monoids and hence $(R,\vee,\wedge,0,1)$ is a bounded lattice. According to (R4), $\alpha$ is an antitone involution on this bounded lattice. If $a\leq b$ then
\[
a\vee(b\wedge\alpha(a))=(b\wedge\alpha(a))\vee a=b\ast a=b\ast(a\wedge b)=b
\]
according to (R5) and (R6).
\end{proof}

Finally, we can show that the just introduced correspondence between orthomodular lattices and coupled right orthosemirings is one-to-one.

\begin{Theorem}\label{th3}
Let ${\mathcal L}=(L,\vee,\wedge,',0,1)$ be an orthomodular lattice. Then ${\bf L}({\bf N}({\mathcal L}))={\mathcal L}$.
\end{Theorem}

\begin{proof}
This is clear.
\end{proof}

\begin{Theorem}\label{th4}
Let ${\mathcal N}=((R,\vee,\cdot,0,1),(R,\wedge,\ast,1,0),\alpha)$ be a coupled right orthosemiring. Then ${\bf N}({\bf L}({\mathcal N}))={\mathcal N}$.
\end{Theorem}

\begin{proof}
Let ${\bf N}({\bf L}({\mathcal N}))=((R,\vee,\odot,0,1),(R,\wedge,\oplus,1,0),\alpha)$ and $a,b\in R$. Then
\[
a\oplus b=(a\wedge\alpha(b))\vee b=a\ast b
\]
according to (R5). From this and (R4) it finally follows
\[
a\odot b=\alpha(\alpha(a)\oplus\alpha(b))=\alpha(\alpha(a)\ast\alpha(b))=ab.
\]
\end{proof}

In view of the previous results the question arises if similar constructions would be possible for other types of algebras. We believe that this is true, but the question arises how far a meaningful generalization of the notion of a semiring could go. We think that (right) distributivity of $\cdot$ with respect to $+$ is an essential property which cannot be omitted. Hence corresponding constructions for algebras like Heyting algebras do not look realistic in our opinion.

Authors' addresses:

Ivan Chajda \\
Palack\'y University Olomouc \\
Faculty of Science \\
Department of Algebra and Geometry \\
17.\ listopadu 12 \\
77146 Olomouc \\
Czech Republic \\
ivan.chajda@upol.cz

Helmut L\"anger \\
TU Wien \\
Faculty of Mathematics and Geoinformation \\
Institute of Discrete Mathematics and Geometry \\
Wiedner Hauptstra\ss e 8-10 \\
1040 Vienna \\
Austria \\
helmut.laenger@tuwien.ac.at
\end{document}